\documentclass{amsart}
\usepackage{enumerate}
\usepackage{graphs}
\usepackage{graphicx}

\renewcommand{\epsilon}{\varepsilon}
\renewcommand{\phi}{\varphi}

\newcommand{\phixnull}{\phi_{x_0}}
\newcommand{\Dn}{{D_0}}
\newcommand{\Mtilde}{\tilde {M}}
\newcommand{\deltatilde}{{\tilde{\delta}}}

\newcommand{\la}{\langle}
\newcommand{\ra}{\rangle}

\begin{document}
\title{Monotone unitary families}

\author{Daniel Grieser}
\address{Institut f\"ur Mathematik, Carl von Ossietzky Universit\"at Oldenburg, D-26111 Oldenburg}
\email{grieser@mathematik.uni-oldenburg.de}
\keywords{Perturbation theory, spectral theory}
\subjclass[2000]{Primary
             47A55 
                         }
\begin{abstract}
A unitary family is a  family of unitary operators $U(x)$ acting on a finite dimensional hermitian vector space, depending analytically on a real parameter $x$. It is monotone if $\frac1i U'(x)U(x)^{-1}$ is a positive operator for each $x$. We prove a number of results generalizing standard theorems on the spectral theory of a single unitary operator $U_0$, which correspond to the 'commutative' case $U(x)=e^{ix}U_0$. Also, for a two-parameter unitary family -- for which there is no analytic perturbation theory -- we prove an implicit function type theorem for the spectral data under the assumption that the family is monotone in one argument.
\end{abstract}
\maketitle

\section{Introduction}
Let $U(x)$ be a family of unitary operators on a Hermitian vector space $V$ of
dimension $M<\infty$, depending real analytically on $x\in\R$ (or an interval in
$\R$). Then
\begin{equation}
\label{eqn monotone}
D(x) := \frac 1 i U'(x) U(x)^{-1}
\end{equation}
is symmetric. Here $U'(x)$ is the derivative with respect to $x$. We call  $U$ \emph{monotone} if $D(x)$ is a positive operator for all $x$.
Denote
$$ W(x)=\Ker (I-U(x))\quad\text{ and } \calZ = \{x:\, W(x)\neq \{0\}\}.$$
Thus $x\in\calZ$ iff $U(x)$ has eigenvalue one.

A model case for this setup is $U(x)=e^{ix}U_0$ for a unitary $U_0$. Then  $W(x)$ is the eigenspace of $U_0$ with
eigenvalue $e^{-ix}$. Standard facts from the spectral theory of $U_0$ may be restated in terms of $\calZ$ and $W(x)$, for example:
\begin{itemize}
\item
$\calZ$ is a $2\pi$-periodic sequence, having exactly $M$ terms in each half-open interval of length $2\pi$ (counting 'multiplicities').
\item
If $I$ is an interval of length less than $2\pi$ then the spaces $W(x)$, $x\in I$, are linearly independent (even pairwise orthogonal).
\item
If $\|(I-U(x_0))\phi \| \leq \epsilon\|\phi\|$ for some $\phi\in V\setminus\{0\}$, $x_0\in\R$ and $\epsilon\geq 0$ then $e^{-ix_0}$ lies within distance $\epsilon$ of an eigenvalue of $U_0$ (and so $\dist(x_0,\calZ)\leq\pi\epsilon/2$). Furthermore, if $\epsilon'>\epsilon$ and $P$ denotes the orthogonal projection to $\bigoplus_{x:|e^{-ix}-e^{-ix_0}|<\epsilon'} W(x)$ then
\begin{equation}
\label{eqn Pphi estimate}
\|\phi-P\phi\| \leq \frac\epsilon{\epsilon'} \|\phi\|.
\end{equation}
\end{itemize}

In this paper we prove generalizations of these facts to arbitrary monotone unitary families, see Theorem \ref{thm counting ev} in Section \ref{sec counting ev}
and Theorems \ref{thm direct sum W} and \ref{thm stability} in Section \ref{sec eigenspaces}. The estimates are expressed in terms of uniform bounds on the first and second derivatives of $U$: Assume $\dmin,\dmax,d_2>0$ are such that
\begin{equation}
\label{eqn dminmax def}
\dmin I\leq D(x) \leq \dmax I, \quad \|U''(x)\| \leq d_2\quad\text{ for all }x.
\end{equation}
Of course, such constants exist always locally. However, in applications of our results, see \cite{Gri:SGNS}, it is essential that the estimates are uniform in terms of \eqref{eqn dminmax def}, and do not depend on additional data like separation of elements of $\calZ$, see the explanation below.

An important difference between a general unitary family and the model case
is that $U(x)$ and $U(x')$ do not in general commute for $x\neq x'$. A consequence of this is that, while one may take a logarithm of $U$, i.e. find an analytic family $A(x)$ of symmetric operators such that
\begin{equation}
\label{eqn def A}
U(x) = e^{iA(x)}\quad\text{ for all }x,
\end{equation}
it is usually not true that $D(x)$ equals $A'(x)$, or even that positivity of $D(x)$ implies positivity of $A'(x)$. The opposite implication is true, however. See Section \ref{secpositive}.

Finally, we prove a result on two-parameter perturbation theory. Recall the main result of one-parameter perturbation theory (see \cite{Rel:SSI}, \cite{Kat:PTLO}): For an analytic family $U(x)$ of unitary operators on $V$, there are real analytic functions $\mu_j(x)$, $\phi_j(x)$ for $j=1,\dots,M=\dim V$ having values in $\R$ and $V$, respectively, such that for each $x$ the eigenvalues of $U(x)$ are $e^{i\mu_j(x)}$, with corresponding orthonormal basis
of eigenvectors $\phi_j(x)$:
\begin{equation}
\label{eqn def muj phij}
U(x)\phi_j(x) = e^{i\mu_j(x)}\phi_j(x).
\end{equation}
It is well-known that the analogous statement for two-parameter families of operators is false in general. However, we prove a related implicit function type theorem for the spectral data of a two-parameter family for which the dependence on one parameter is monotone.
It may be regarded as the natural unitary family generalization of the one-parameter perturbation theory. See Theorem \ref{thm unitary perturb} in Section \ref{sec two parameters}.

The analytic functions $\mu_j$ and $\phi_j$ play a central role in the proofs of our theorems. It is essential to control their derivatives. For  $\mu_j$ this is easy from \eqref{eqn dminmax def}. However, $\phi_j$ may vary wildly whenever $e^{i\mu_j}$ is very close to another eigenvalue.
To control this variation is the main technical problem in the proof of the generalization of \eqref{eqn Pphi estimate}, Theorem \ref{thm stability}. Note that, unlike in the model case, one may not assume that the $e^{i\mu_j(x)}$ for fixed $x$ but varying $j$ are uniformly separated, or equivalently that the elements of $\calZ$ are uniformly separated. This can already be seen in the simple example $U(x)=e^{ixL}$, where $L$ is a diagonal matrix with positive diagonal entries that are independent over $\Q$: If $M\geq 2$ then for any $\epsilon>0$ there are $x,x'\in\calZ$ satisfying $0<|x-x'|<\epsilon$.
\medskip

The problems we study here arose in the context of a singular perturbation problem: In \cite{Gri:SGNS} we study the eigenvalues and eigenfunctions of the Laplacian on a space $X^N$ which has a fixed compact part connected by cylindrical  necks of length $N>0$, and in particular their asymptotic behavior as $N\to\infty$. The unitary families arise from the scattering matrix of the limit problem (infinitely long 'necks').

\section{Eigenvalue distribution} \label{sec counting ev}
\begin{theorem}
\label{thm counting ev}
Let $U$ be a monotone unitary family on $\R$. Then
$\calZ\subset\R$ is a discrete subset, and more precisely for all $A<B$
\begin{equation}
\label{eqn asymp}
\left| \sum\limits_{x:A< x < B} \dim W(x) - \frac 1 {2\pi}\int_A^B \tr D(x)\, dx
\right| < M (:=\dim V)
\end{equation}
\end{theorem}
In the special case $U(x)=e^{ixL}U_0$, $L>0$, this implies the asymptotics
$$ \sum\limits_{x:0< x < B} \dim W(x) \sim \frac{\tr L}{2\pi} B,\quad B\to\infty.$$ This is the Weyl asympotitcs of a quantum graph, see for example \cite{GnuSmi:QGAQCUSS}. We give a much simpler proof than \cite{GnuSmi:QGAQCUSS}.

First, we differentiate \eqref{eqn def muj phij} and obtain monotonicity of the functions $\mu_j$ from monotonicity of $U$:
\begin{equation}
\label{eqn muj monotone}
\mu_j' = (\frac 1i U' U^{-1}\phi_j,\phi_j) = (D\phi_j,\phi_j) > 0.
\end{equation}

Also, for each $x$ we have
\begin{equation}
\label{eqn Wx characterization}
W(x)=\Span\{\phi_j(x):\, \mu_j(x) \in 2\pi\Z\}
\end{equation}

\begin{proof}[Proof of Theorem \ref{thm counting ev}]
We have
\begin{align*}
\sum_j \left(\mu_j(B) - \mu_j (A)\right) &= \int_A^B \sum_j \la
D(x)\phi_j(x),\phi_j(x)\ra\, dx \\
                                         &= \int_A^B \tr D(x)\, dx
\end{align*}

Since the $\mu_j$ are strictly increasing, we get from \eqref{eqn Wx
characterization}
\begin{equation*}
\sum_{x: A<x<B} \dim W(x) = \sum_j \#\{k\in\Z:\, \mu_j(A) < 2\pi k < \mu_j(B)\} =
\sum_j \left(\frac{\mu_j(B)-\mu_j(A)}{2\pi} + R_j\right)
\end{equation*}
with $|R_j|<1$, and this gives \eqref{eqn asymp}.
\end{proof}

\section{Eigenspaces} \label{sec eigenspaces}
In this section we consider monotone unitary families satisfying the estimates \eqref{eqn dminmax def}. By \eqref{eqn muj monotone} we have
\begin{equation}
\label{eqn mu monotone}
\dmin \leq \mu_j'(x) \leq \dmax\quad \text{ for all }x \text{ and }j=1,\dots,M.
\end{equation}
First, we have independence of eigenspaces.
\begin{theorem}
\label{thm direct sum W}
Let $U$ be a monotone unitary family satisfying \eqref{eqn dminmax def}.
Let $I$ be an interval of length at most $\frac {2\dmin}{d_2 M}$. Then the spaces
$W(x),$ $x\in I$, are independent, i.e.
\begin{equation}
\label{eqn Wx independent}
\text{If }\phi_x\in W(x)\text{ for each }x\in I\cap\calZ\text{ and }\sum_x
\phi_x=0\text{ then }\phi_x=0\ \forall x.
\end{equation}
\end{theorem}
The following theorem gives a stable version of almost orthogonality.
\begin{theorem} \label{thm stability}
Let $U$ be a monotone unitary family satisfying \eqref{eqn dminmax def}.
Assume $\phi\in V\setminus 0$ satisfies
\begin{equation}
\label{eqn phiest}
\|(I-U(x_0))\phi\| \leq \epsilon \|\phi\|.
\end{equation}
Then
\begin{equation}
\label{eqn dist to Z}
\dist (x_0,\calZ) \leq \frac\pi 2\frac {\epsilon}{\dmin}.
\end{equation}
Furthermore, there is a constant $C$ only depending on $\dmin,\dmax,d_2,M$ such that the following holds: Suppose $0<\epsilon<\epsilon'<C^{-1}$ and ${\epsilon}/\epsilon'< C^{-1}$. Denote by $P_W$ the orthogonal projection to
$W=\bigoplus_{|x-x_0|\leq\epsilon'} W(x)$. Then
\begin{equation}
\label{eqn dist to sum of W}
\|\phi - P_W\phi\| \leq C\left(\frac\epsilon{\epsilon'}\right)^{\frac1{M+1}}\|\phi\|.
\end{equation}
\end{theorem}

In the proofs of Theorems \ref{thm direct sum W} and \ref{thm stability} we will
need the following estimate, which replaces orthogonality of the eigenspaces of a
unitary operator.
\begin{lemma}
\label{lemma almost orthogonal}
If $\phi\in W(x)$, $\psi\in W(y)$ and $x\neq y$ then
\begin{equation}
\label{eqn almost orthogonal}
|\la D(x)\phi,\psi\ra| \leq \frac{d_2}2 |x-y|\cdot\|\phi\|\cdot\|\psi\|.
\end{equation}
\end{lemma}
\begin{proof}
By Taylor's formula,
$$ U(y)U(x)^{-1} = I  + (y-x) U'(x)U(x)^{-1} + (y-x)^2 R,\quad \|R\|\leq
\frac{d_2}2,$$
so if $U(x)\phi=\phi$, $U(y)\psi=\psi$ then
\begin{align*}
\la \phi,\psi\ra & = \la U(x)^{-1}\phi, U(y)^{-1}\psi\ra = \la
U(y)U(x)^{-1}\phi,\psi\ra\\
            & = \la\phi,\psi\ra + i(y-x)\la D(x)\phi,\psi\ra + (y-x)^2 \la
            R\phi,\psi\ra,
\end{align*}
and this gives \eqref{eqn almost orthogonal}.
\end{proof}

\begin{proof}[Proof of Theorem \ref{thm direct sum W}]
Let $\phi_x\in W(x)$ for $x\in I\cap\calZ$, and assume $\sum_x \phi_x=0$.
Let $\phi_{x_0}$ have maximal norm among all $\phi_x$. Then $0=\la
D(x_0)\phi_{x_0},\sum_x\phi_x\ra$ gives with \eqref{eqn dminmax def} and \eqref{eqn
almost orthogonal}
\begin{align*}
\dmin \|\phixnull\|^2 & \leq \la D(x_0)\phixnull,\phixnull\ra = \left|\sum_{x\neq
x_0} \la D(x_0),\phixnull,\phi_x\ra\right| \\
 & \leq \sum_{x\neq x_0} \frac{d_2}2 |I|\cdot \|\phixnull\|\cdot\|\phi_x\|
 \leq (M-1) \frac{d_2}2 |I|\cdot \|\phixnull\|^2,
\end{align*}
so if $\dmin > (M-1)\frac{d_2}2 |I|$ then $\phixnull=0$ and hence $\phi_x=0$ for all
$x$. This implies the claim.
\end{proof}

\begin{proof}[Proof of Theorem \ref{thm stability}]
The first estimate follows easily from the fact that, by the lower bound in
\eqref{eqn mu monotone}, an eigenvalue close to one of $U(x_0)$ will turn into an
eigenvalue equal to one of $U(x)$, for some $x$ close to $x_0$:
Let $B(x)=I-U(x)$ and let $\lambda_j(x)=1-e^{i\mu_j(x)}$ be the eigenvalues of
$B(x)$. The assumption \eqref{eqn phiest} implies that
$|\lambda_j(x_0)|\leq\epsilon$ for some $j$, and this implies
$\dist(\mu_j(x_0),2\pi\Z) \leq \frac\pi 2 \epsilon$, and then
$\mu_j'\geq\dmin$ shows that there is an $x$ satisfying $|x-x_0|<\pi\epsilon/2\dmin$
and $\mu_j(x)\in2\pi\Z$, hence $x\in\calZ$, so \eqref{eqn dist to Z} follows.

For $\delta>0$ let $P_\delta(x)$ denote the orthogonal projection to the sum of the
eigenspaces of $B(x)$ with eigenvalues $|\lambda_j(x)|\leq\delta$. Then
$\|B(x_0)\phi\|\leq \epsilon \|\phi\|$ implies
\begin{equation}
\label{eqn phiPest1}
\|\phi - P_\delta(x_0)\phi\| \leq \frac\epsilon\delta \|\phi\|
\end{equation}
(see \eqref{eqn Pphi estimate}, which also applies to normal operators).
To make this a good estimate, we want to take $\delta>>\epsilon$.
Our goal is to replace $P_\delta(x_0)$ by $P_W$ here. The idea is that eigenspaces
of $B(x_0)$ with eigenvalue $|\lambda_j(x_0)|\leq\delta$ will turn into nullspaces
of $B(x)$ for some $x$ within $2\delta/\dmin$ of $x_0$, by the first part of this
proof. However, the variation of eigenspaces is much less well behaved than the
variation of eigenvalues: An eigenspace may change rapidly with $x$ if the
eigenvalue is very close to another eigenvalue. Therefore, we need to consider not
single eigenspaces but rather clusters of eigenspaces.

The variation of eigenspaces is given as follows (see \cite{Kat:PTLO}): Fix $x$. If
$B(x)$ has no eigenvalue on the circle $|\lambda|=\delta$ then, with a prime
denoting derivative in $x$,
\begin{equation}
\label{eqn variation spectral proj}
P_\delta ' = \sum_{j:|\lambda_j|<\delta} \, \sum_{k:|\lambda_k|>\delta} \frac 1
{\lambda_j-\lambda_k} (P_jB'P_k + P_kB'P_j).
\end{equation}
Here, all quantities are evaluated at $x$, and $P_j$ is the orthogonal projection to
$\Span\phi_j$.
Taking norms and using orthogonality of the $P_j$ one obtains from this, using
$\|B'\|\leq\dmax$,
\begin{equation}
\label{eqn P'estimate}
\|P_\delta '\| \leq \dmax M\left( \min_{|\lambda_j|<\delta,\,|\lambda_k|>\delta}
|\lambda_j-\lambda_k|\right)^{-1}
\end{equation}
We need to choose $\delta$ carefully to make the spectral gap not too small: Let
$s=(\epsilon'/\epsilon)^{1/(M+1)}$ and consider the $M$ disjoint subintervals
$$ [\epsilon s^k,\epsilon s^{k+1})\quad\text{ for }k=1,\dots,M$$
of $(\epsilon,\epsilon')$.
Since $B(x_0)$ has $M$ eigenvalues and one of them has absolute value
$\leq\epsilon$, at least one of these intervals contains no $|\lambda_j(x_0)|$.
Assume
\begin{equation}
\label{eqn delta assn}
[\delta,\delta s),\ \delta=\epsilon s^k,\ \text{ contains no } |\lambda_j(x_0)|.
\end{equation}
We then have:
\begin{enumA}
\item
The eigenvalues of $B(x_0)$ with $|\lambda_j(x_0)|\leq \delta$ are in 1-1
correspondence with those $x\in\calZ$ (counted with multiplicity $\dim W(x)$)
satisfying
$$ |x-x_0| \leq \delta' := \frac {2\delta}{\dmin}.$$
(Proof: Each such eigenvalue turns into a zero of $\lambda_j(x)$ for such an $x$, by
the argument at the beginning of this proof. Conversely, if $\lambda_j(x)=0$ then
$|\lambda_j(x_0)|\leq |x-x_0|\dmax$ since $|\lambda_j'|=\mu_j'\leq\dmax$ and hence
$|\lambda_j(x_0)|\leq 2\delta\frac\dmax\dmin < \delta s$ provided
$\epsilon/\epsilon'$ is sufficiently small (and therefore $s$ big),
and by \eqref{eqn delta assn} this implies further $|\lambda_j(x_0)| < \delta$.)
\item
The smaller interval $(\delta+2\delta\frac\dmax\dmin,\delta s -
2\delta\frac\dmax\dmin)$ contains no $|\lambda_j(x)| $ for any $x$ with $|x-x_0|
\leq \delta'$.
(Follows directly from $|\lambda'_j|\leq\dmax$.)
\end{enumA}
The length of the interval in B) is $\delta(s-1-4\frac\dmax\dmin)$ which is $\geq
\delta s/2>0$ if $s$ is sufficiently big.
Choose $\deltatilde$ in this interval, then we get from \eqref{eqn P'estimate}
$\|P_\deltatilde'(x)\| \leq \dmax M \left(\delta s/2\right)^{-1}$ for
$|x-x_0|\leq\delta'$.
Integration gives
\begin{equation}
\label{eqn diff Pdelta}
\|P_\deltatilde(x)- P_\deltatilde(x_0)\| \leq \epsilon_1 := \frac4s \frac\dmax\dmin
M
\quad\text{ for }|x-x_0|\leq\delta'.
\end{equation}
This implies
\begin{equation}
\label{eqn projest}
\|\psi - P_\delta(x_0)\psi\| \leq \epsilon_1 \|\psi\|
\quad\text{ for } x\in\calZ,\ |x-x_0|\leq\delta',\ \psi\in W(x)
\end{equation}
since $P_\deltatilde(x)\psi=\psi$ then and $P_\deltatilde(x_0) = P_\delta(x_0)$.

Next we want to extend this estimate to $\psi\in W':=\bigoplus_{|x-x_0|\leq\delta'}
W(x)$. For this it is essential that, by \eqref{eqn almost orthogonal} and the
positive definiteness of $D(x)$, the angles between different $W(x)$ are bounded
away from zero. To carry this out, we first derive from \eqref{eqn almost
orthogonal} an estimate where all $D(x)$ are replaced by $D(x_0)$: From
$D'=U''U^{-1} + D^2$ we have $\|D'\| \leq d_2+\dmax^2$; integration yields
$\|D(x_0)\|\leq \|D(x)\|+(d_2+\dmax^2)\delta'$ for $|x-x_0|\leq\delta'$, and then
\eqref{eqn almost orthogonal} gives, with $D_0:= D(x_0)$,
\begin{equation}
\label{eqn almost orthogonal 2}
|\la D_0\psi_x,\psi_y\ra| \leq \delta'' \|\psi_x\|\cdot\|\psi_y\|,\quad
\delta'':=\delta'(d_2+\dmax^2+d_2)
\end{equation}
for $\psi_x\in W(x)$, $\psi_y\in W(y)$, $x\neq y$ and $|x-x_0|\leq\delta'$,
$|y-x_0|\leq\delta'$.
Introduce the scalar product $(\phi,\psi)_\Dn :=\la D_0\phi,\psi\ra$ on $V$, with
norm $\|\psi\|_\Dn = \sqrt{\la D_0\psi,\psi\ra}$, then
\begin{equation}
\label{eqn Dnorm est}
\dmin \|\psi\|^2 \leq \|\psi\|^2_\Dn \leq \dmax \|\psi\|^2
\end{equation}
so \eqref{eqn almost orthogonal 2} gives
\begin{equation}
\label{eqn almost orthogonal 3}
|(\psi_x,\psi_y)_\Dn| \leq \frac{\delta''}{\dmin} \|\psi_x\|_\Dn\cdot\|\psi_y\|_\Dn
\end{equation}
for the same $\psi_x,\psi_y$ as there. By simple standard calculations this implies
\begin{equation}
\label{eqn ao 4}
\|\sum_x\psi_x\|_\Dn^2 \geq \left(1-\frac{\delta''}\dmin (\Mtilde - 1)\right) \sum_x
\|\psi_x\|_\Dn^2
\end{equation}
where the sums are over all $x\in\calZ$ with $|x-x_0|\leq\delta'$, $\psi_x\in W(x)$
are arbitrary and $\Mtilde$ is the number of summands. Now by A) above $\Mtilde\leq
M$.
Since $\delta<\epsilon'$ the expression in parantheses is $\geq \frac12$ for sufficiently small $\epsilon'$,
so \eqref{eqn ao 4} gives $\sum_x \|\psi_x\|^2_\Dn \leq 2\|\sum_x \psi_x\|_\Dn^2$,
which with \eqref{eqn Dnorm est} gives $\sum_x \|\psi_x\|^2 \leq
2\frac\dmax\dmin\|\sum_x \psi_x\|^2$ and so
\begin{equation}
\label{eqn sum est}
\sum_x \|\psi_x\| \leq \sqrt{2M\frac\dmax\dmin}\, \|\sum_x\psi_x\|.
\end{equation}
We return to \eqref{eqn projest}. If $\psi\in W'=\bigoplus_{|x-x_0|\leq\delta'}
W(x)$, $\psi=\sum_x\psi_x$ then we get from \eqref{eqn sum est}
\begin{align}
\|\psi - P_\delta(x_0)\psi\| &\leq \sum_x \|\psi_x - P_\delta(x_0)\psi_x\|\leq
\epsilon_1 \sum_x \|\psi_x\|\notag\\
\label{eqn projest2}
& \leq \epsilon_2\|\psi\|,\quad \epsilon_2:=\epsilon_1\sqrt{2M\frac\dmax\dmin}
\end{align}
This means that $W'$ is close to $W_0:=\Ran P_\delta(x_0)$. Now by A) above, $W'$
and $W_0$ have the same dimension, so this implies by standard arguments  that $W_0$ is also close to $W'$, more precisely,
with $P_{W'}:V\to W'$ the orthogonal projection,
\begin{equation}
\label{eqn proj close}
\|\psi - P_{W'}\psi \| \leq \epsilon_2 \|\psi\|,\quad \psi\in W_0.
\end{equation}
Finally, assume that $\phi\in V$ satisfies \eqref{eqn phiest}. Then \eqref{eqn
phiPest1} with $\delta$ as above together with \eqref{eqn proj close} give, again by
standard facts about distances of subspaces,
\begin{equation}
\label{eqn PW' est}
 \|\phi - P_{W'}\phi \| \leq (\epsilon_2 + \frac\epsilon\delta) \|\phi\|.
\end{equation}
Putting everything together, we have $\delta=\epsilon s^k$ for some
$k\in\{1,\dots,M\}$ with $s=(\epsilon'/\epsilon)^{1/(M+1)}$, which implies $\epsilon/\delta \leq s^{-1}$ and
$\delta'\leq \epsilon'$.
Also, $\epsilon_1=\frac4s \frac\dmax\dmin$
and
$\epsilon_2 = \epsilon_1\sqrt{2M\frac\dmax\dmin}$. Altogether, the right hand side
of \eqref{eqn PW' est} is bounded by $Cs^{-1}\|\phi\|$, and since the left hand
side only decreases when replacing $W'$ by the bigger space
$\bigoplus_{|x-x_0|\leq\epsilon'} W(x)$, the Theorem is proven.
\end{proof}

\section{Monotonicity of $U$ and of its logarithm}\label{secpositive}
Denote by $\calS(V)$ and $\calU(V)$ the spaces of symmetric resp. unitary operators on $V$. The map  $\calS (V)\to \calU(V)$, $A\mapsto e^{iA}$  has non-singular differential everywhere and is surjective, so it is a covering map. Hence any curve $U:x\mapsto U(x)$ in $\calU(V)$ may be lifted to a curve $x\mapsto A(x)$ in $\calS(V)$ (that is, $U(x)=e^{iA(x)}$ for all $x$) and the lift is unique if one prescribes it for one value of $x$. Furthermore, the lifted curve is analytic if $U$ is.

\begin{proposition}
\label{prop monotone log}
Let $x\mapsto A(x)$ be a $C^1$ family of symmetric operators and $U(x)=e^{iA(x)}$. Then
\begin{equation}
\label{eqn U' A' relation}
\frac1i U' U^{-1} = \int_0^1 e^{i\tau A} A' e^{-i\tau A}\,d\tau.
\end{equation}
Here, a prime denotes differentiation with respect to $x$, and $U,U',A,A'$ are taken at a fixed $x$.
\end{proposition}
\begin{proof}
Let $W(t,x) = \frac1i \frac\partial{\partial x} e^{itA(x)}$. Then
$$ \frac\partial{\partial t}W = \frac1i \frac\partial{\partial x}\frac\partial{\partial t} e^{itA} = \frac\partial{\partial x} Ae^{itA}
= A'e^{itA} + iA W. $$
Now fix $x$, and let $B(t)=A'e^{itA}$. The solution of the ordinary differential equation $\frac d{dt} Y (t)= B(t) + iA Y(t)$ with $Y(0)=0$ is
$$ Y(t) = \int_0^t e^{i(t-\tau)A}B(\tau)\, d\tau$$
as can be verified directly. (This is called Duhamel's principle.)
Now $W(\cdot,x)$ satisfies the same first order ODE and initial condition as $Y$, so it follows that $W(t,x)=Y(t)$ for all $t$. Now rewrite $Y(t)=\int_0^t e^{i\tau A}B(t-\tau)\,d\tau$ and set $t=1$ to obtain \eqref{eqn U' A' relation}.
\end{proof}

\begin{corollary}
Let $x\mapsto A(x)$ be a $C^1$ family of symmetric operators. If $A'(x)>0$ for each $x$ then the unitary family $U(x)=e^{iA(x)}$ is monotone.
\end{corollary}
\begin{proof}
Positivity of $A'$ implies positivity of $e^{i\tau A}A'e^{-i\tau A}$ for each $\tau$, so the claim follows from \eqref{eqn U' A' relation}.
\end{proof}
The converse is not true. As an example let
$A_0 = \begin{pmatrix}
0 & -\pi \\ \pi & 0
\end{pmatrix}$,
$B= \begin{pmatrix}
-b & 0 \\ 0 & 1
\end{pmatrix}
$ with $0<b<1$
and $A(x) = A_0+xB$. Then $e^{i\tau A_0} = \begin{pmatrix}
\cos \pi \tau & -\sin \pi \tau \\ \sin \pi \tau & \cos \pi \tau
\end{pmatrix}
$
is rotation by $\pi \tau$, and a short calculation shows that
$\int_0^1 e^{i\tau A_0}B e^{-i\tau A}\,d\tau = \frac{1-b}{2} I$ (this is also clear without calculation since the result must be rotation invariant with trace equal to $\tr B = 1-b$; in essence, the negative direction of $B$ gets averaged away against the positive direction). Therefore, $U(x)=e^{iA(x)}$ is monotone near $x=0$ but $A'(0)=B$ is not positive.

\section{Two parameter families}\label{sec two parameters}
\begin{theorem}
\label{thm unitary perturb}
Let $U(x,y)$ be a unitary operator in a finite-dimensional Hermitian vector space
depending real analytically on $x,y\in\R$. Assume
\begin{equation}
\label{eqn positivity}
\frac 1i \frac{\partial U}{\partial x} U^{-1} > 0\quad \text{ at } (x_0,y_0).
\end{equation}
Then the set $\{(x,y):\, U(x,y)\text{ has eigenvalue one}\}$ is, in a neighborhood
of $(x_0,y_0)$, a union of real analytic curves $x=x_j(y)$. The corresponding
projections $P_j(y)$ to the eigenspace of $U(x_j(y),y)$ with eigenvalue one are also
analytic functions of $y\neq y_0$, extending analytically to $y=y_0$, and $\sum_j
P_j(y_0)$ is the projection to $\ker (I-U(x_0,y_0))$.
\end{theorem}
Note that in general it is not true that the eigenvalues and eigenprojections of
$U(x,y)$ may be arranged as real analytic functions of $(x,y)$, see \cite{Kat:PTLO},
II.6.1. While the example given there (in the analogous case of self-adjoint
operators) does not satisfy the positivity assumption \eqref{eqn positivity}, it can
be easily modified so it does, by adding a multiple of the identity. Explicitly, one
may take $A(x,y)=\begin{pmatrix}
3x & y \\ y & x
\end{pmatrix}$ and $U(x,y)=e^{iA(x,y)}$ and $(x_0,y_0)=(0,0)$.

Note also that the statement of the theorem reduces to the well-known facts of one-parameter perturbation theory in case $U(x,y)=e^{ix}U(y)$, for an analytic one-parameter family of unitary operators $U(y)$.
\begin{proof}
Let w.l.o.g.\ $x_0=y_0=0$.

We first consider the case $U(0,0)=I$. Let $A=\frac1i \log U$ near $(x,y)=(0,0)$.
Then the operators $A(x,y)$ are self-adjoint, $A(0,0)=0$, and  $\partial A/\partial
x (0,0)> 0$ since it equals $\frac1i\frac{\partial U}{\partial x}U^{-1}(0,0)$ by \eqref{eqn U' A' relation}, and we need to prove that the set $S=\{(x,y):\, A(x,y)\text{ is not
invertible }\}$ is a union of real analytic curves as claimed.\footnote{I am grateful to Y. Colin-de-Verdi\`ere for a fruitful
discussion on this}

If $A(x,y)=xA+yB$ is linear in $x,y$, then (since $A>0$) $A$ and $B$ may be
diagonalized simultaneously, hence may be assumed to be diagonal, and then it is
obvious that $S$ is a union of lines, $x_j(y)=y b_j/a_j$, where $a_j,b_j$ are the
diagonal entries of $A,B$, respectively.
In general, write $A(x,y)=xA+yB + C(x,y)$ with $C(x,y)=O(|x,y|^2)$ and w.l.o.g.\
$A,B$ diagonal. Then, if the dimension of the vector space is $M$,
$$\det A(x,y) = \prod_{j=1}^M (xa_j + yb_j) + O(|x,y|^{M+1}),$$
and a standard argument (using polar coordinates) shows that the zero set of this
function is a union of real analytic lines
$x=x_j(y)$, having tangents $xa_j+yb_j=0$ at the origin.

If $U(0,0)$ is arbitrary, let $W=U-I$ (where $I$ denotes the identity) and $V_0=\Ker
W(0,0)$ and $V_1$ its orthogonal complement. Let $W_{kl}(x,y)$, $k,l=0,1$, be the
'submatrices' of $W(x,y)$ corresponding to the decomposition $V_0\oplus V_1$. Then
$W_{00}$, $W_{01}$ and $W_{10}$ vanish at $(x,y)=(0,0)$, and $W_{11}$ is invertible
at $(0,0)$ and hence in a neighborhood. Then the equation $W(v_0\oplus v_1)=0$,
where $v_0\in V_0$, $v_1\in V_1$,  is equivalent to $W'v_0=0$, where
$W'=W_{00}-W_{01}W_{11}^{-1}W_{10}$, and  $v_1=-W_{11}^{-1}W_{10}v_0$.
Therefore, $U(x,y)$ has eigenvalue one iff the operator $U'(x,y)=W'(x,y)+I_{V_0}$ on
$V_0$ has eigenvalue one. One easily checks that $U'(x,y)$ is unitary.
Since $W'(0,0)=0$ the claim now follows from the case considered first.

Let $C_j(y)=U(x_j(y),y)$ and let $P_j(y)$ be the projection to $\Ker C_j(y)$. Since
$C_j$ is analytic in $y$, its eigenprojections are analytic for $y\neq 0$ (but near
zero) and extend analytically to $y=0$ (see \cite{Kat:PTLO}), so this is in
particular true for $P_j$.
\end{proof}

\bibliographystyle{amsplain}
\bibliography{mypapers,dglib,mathlib}

\end{document}